\documentclass[11pt]{article}
\usepackage{amscd, amsmath, amssymb, amsthm}
\usepackage[all,cmtip]{xy}
\usepackage[pagebackref]{hyperref}

\title{Essential dimension of the spin groups in characteristic 2}
\author{Burt Totaro}
\date{  }

\def\Z{\text{\bf Z}}

\def\P{\text{\bf P}}
\def\F{\text{\bf F}}

\DeclareMathOperator{\et}{et}

\def\arrow{\rightarrow}
\def\inj{\hookrightarrow}

\DeclareMathOperator{\Br}{Br}

\DeclareMathOperator{\Spin}{Spin}
\DeclareMathOperator{\HSpin}{HSpin}
\DeclareMathOperator{\ed}{ed}
\DeclareMathOperator{\ind}{ind}
\DeclareMathOperator{\PGO}{PGO}

\begin{document}
\maketitle
\newtheorem{theorem}{Theorem}[section]
\newtheorem{proposition}[theorem]{Proposition}
\newtheorem{corollary}[theorem]{Corollary}
\newtheorem{lemma}[theorem]{Lemma}

\theoremstyle{definition}
\newtheorem{definition}[theorem]{Definition}
\newtheorem{example}[theorem]{Example}

\theoremstyle{remark}
\newtheorem{remark}[theorem]{Remark}

The essential dimension of an algebraic group $G$ is a measure
of the number of parameters needed to describe
all $G$-torsors over all fields. A major achievement
of the subject was the calculation of the essential
dimension of the spin groups over a field of characteristic not 2,
started by Brosnan, Reichstein, and Vistoli, and completed
by Chernousov, Merkurjev, Garibaldi, and Guralnick
\cite{BRV, CM, GG}, \cite[Theorem 9.1]{MerkBull}.

In this paper, we determine the essential dimension of the spin group
$\Spin(n)$ for $n\geq 15$
over an arbitrary field (Theorem \ref{spin}). We find that
the answer is the same in all characteristics. In contrast,
for the groups $O(n)$ and $SO(n)$, 
the essential dimension is smaller in characteristic 2,
by Babic and Chernousov \cite{BC}.

In characteristic not 2, the computation of essential dimension
can be phrased to use
a natural finite subgroup of $\Spin(2r+1)$, namely an
extraspecial 2-group, a central extension of $(\Z/2)^{2r}$ by $\Z/2$.
A distinctive feature of the argument
in characteristic 2 is that the analogous subgroup is
a finite group scheme, a central extension of $(\Z/2)^r\times (\mu_2)^r$
by $\mu_2$, where $\mu_2$ is the group scheme of square roots of unity.

In characteristic not 2,
Rost and Garibaldi computed the essential dimension of $\Spin(n)$
for $n\leq 14$ \cite[Table 23B]{Garibaldi}, where case-by-case
arguments seem to be needed.
We show in Theorem \ref{lowdim} that for $n\leq 10$, the essential
dimension of $\Spin(n)$ is the same in characteristic 2 as in
characteristic not 2. It would be interesting to compute
the essential dimension of $\Spin(n)$
in the remaining cases, $11\leq n\leq 14$
in characteristic 2.

This work was supported by NSF grant DMS-1303105. Thanks
to Skip Garibaldi and Alexander Merkurjev for their suggestions.
Garibaldi spotted a mistake in my previous description of the finite
group scheme in the proof of Theorem \ref{spin}.

\section{Essential dimension}
\label{notation}

Let $G$ be an affine group scheme of finite type
over a field $k$. Write $H^1(k,G)$ for the set of isomorphism
classes of $G$-torsors over $k$ in the fppf topology. For $G$ smooth
over $k$, this is also the set of isomorphism classes
of $G$-torsors over $k$
in the etale topology.

Following Reichstein,
the {\it essential dimension }$\ed(G)$ is the smallest natural number $r$
such that for every $G$-torsor $\xi$ over an extension field $E$ of $k$,
there is a subfield $k\subset F\subset E$ such that $\xi$ is isomorphic
to some $G$-torsor over $F$ extended to $E$, and $F$ has transcendence degree
at most $r$ over $k$. (It is essential that $E$ is allowed to be
any extension field of $k$, not just an algebraic extension field.)
There are several survey articles on essential dimension, including
\cite{Reichstein, Merkess}.

For example, let $q_0$ be a quadratic form of dimension $n$ over a field $k$
of characteristic not 2. Then $O(q_0)$-torsors can be identified with
quadratic forms of dimension $n$, up to isomorphism.
(For convenience, we sometimes write $O(n)$ for $O(q_0)$.) Thus the
essential dimension of $O(n)$ measures the number of parameters
needed to describe all quadratic forms of dimension $n$. Indeed,
every quadratic form of dimension $n$
over a field of characteristic not 2 is isomorphic to a diagonal form
$\langle a_1,\ldots,a_n\rangle$. It follows that the orthogonal
group $O(n)$ in characteristic not 2
has essential dimension at most $n$; in fact, $O(n)$
has essential dimension equal to $n$,
by one of the first computations
of essential dimension \cite[Example 2.5]{Reichstein}.
Reichstein also showed that the connected group $SO(n)$ in characteristic
not 2
has essential dimension $n-1$ for $n\geq 3$ \cite[Corollary 3.6]{Reichstein}.

For another example, for a positive integer $n$ and any field $k$,
the group scheme $\mu_n$ of $n$th roots of unity is smooth over $k$
if and only if $n$ is invertible in $k$. Independent of that,
$H^1(k,\mu_n)$ is always isomorphic to $k^*/(k^*)^n$. From that description,
it is immediate that $\mu_n$ has essential dimension at most 1 over $k$.
It is not hard to check that the essential dimension is in fact
equal to 1.

One simple bound is that for any generically free
representation $V$ of a group scheme
$G$ over $k$ (meaning that $G$ acts freely on a nonempty open subset of $V$),
the essential dimension of $G$ is at most $\dim(V)-\dim(G)$
\cite[Proposition 5.1]{MerkBull}. It follows, for example,
that the essential dimension of any affine group scheme of finite
type over $k$ is finite.

For a prime number $p$, the $p$-essential dimension $\ed_p(G)$
is a simplified invariant, defined by ``ignoring field
extensions of degree prime to $p$''. In more detail, for a $G$-torsor
$\xi$ over an extension field $E$ of $k$, define the $p$-essential
dimension $\ed_p(\xi)$ to be the smallest number $r$ such that
there is a finite extension $E'/E$ of degree prime to $p$ such that
$\xi$ over $E'$ comes from a $G$-torsor over a subfield $k\subset F\subset E'$
of transcendence degree at most $r$ over $k$. Then the {\it $p$-essential
dimension }$\ed_p(G)$ is defined to be the supremum of the $p$-essential
dimensions of all $G$-torsors over all extension fields of $k$.

The {\it spin group }$\Spin(n)$ is the simply connected double
cover of $SO(n)$.
It was a surprise when Brosnan, Reichstein, and Vistoli showed
that the essential dimension of $\Spin(n)$ over
a field $k$ of characteristic not 2 is exponentially large,
asymptotic to $2^{n/2}$ as $n$ goes to infinity \cite{BRV}. As an application,
they showed that the number of ``parameters'' needed to describe
all quadratic forms of dimension $2r$ in $I^3$ over all fields
is asymptotic to $2^r$.

We now turn to quadratic forms over a field which
may have characteristic 2.
Define a quadratic form $(q,V)$ over a field $k$ to be {\it nondegenerate }if
the radical $V^{\perp}$ of the associated
bilinear form is 0, and {\it nonsingular }if $V^{\perp}$
has dimension at most 1 and $q$ is nonzero on any
nonzero element of $V^{\perp}$.
(In characteristic not 2, nonsingular and nondegenerate are the same.)
The orthogonal group is defined as the automorphism group scheme
of a nonsingular quadratic form \cite[section VI.23]{KMRT}. For example,
over a field $k$ of characteristic 2, the quadratic form
$$x_1x_2+x_3x_4+\cdots+x_{2r-1}x_{2r}$$
is nonsingular of even dimension $2r$, while the form
$$x_1x_2+x_3x_4+\cdots+x_{2r-1}x_{2r}+x_{2r+1}^2$$
is nonsingular of odd dimension $2r+1$, with $V^{\perp}$ of dimension 1.
The {\it split }orthogonal group over $k$ is the automorphism group
of one of these particular quadratic forms.

Babic and Chernousov computed the essential dimension
of $O(n)$
and the smooth connected subgroup $O^+(n)$ over an infinite field $k$
of characteristic 2 \cite{BC}. (We also write $SO(n)$ for $O^+(n)$
by analogy with the case of characteristic not 2, even though
the whole group $O(2r)$ is contained in $SL(2r)$ in characteristic 2.)
The answer is smaller than in characteristic not 2. Namely,
$O(2r)$ has essential dimension $r+1$ (not $2r$) over $k$. Also,
$O^+(2r)$ has essential dimension $r+1$ for $r$ even,
and either $r$ or $r+1$ for $r$ odd, not $2r-1$.
Finally, the group scheme $O(2r+1)$ has essential dimension $r+2$
over $k$, and $O^+(2r+1)$ has essential dimension $r+1$.
The lower bounds here are difficult, while the upper bounds
are straightforward. For example, to show that $O(2r)$
has essential dimension at most $r+1$ in characteristic 2,
write any quadratic form of dimension $2r$
as a direct sum of 2-dimensional
forms, thus reducing the structure group to $(\Z/2)^r\times (\mu_2)^r$,
and then use that the group $(\Z/2)^r$ has essential dimension only 1
over an infinite field of characteristic 2
\cite[proof of Proposition 13.1]{BC}.

In this paper, we determine the essential dimension of $\Spin(n)$
in characteristic 2 for $n\leq 10$ or $n\geq 15$. Surprisingly,
in view of what happens for $O(n)$ and $O^+(n)$, the results for
spin groups are the same in characteristic 2
as in characteristic not 2. For $n\leq 10$, the lower bound for
the essential dimension is proved by constructing suitable
cohomological invariants. It is not known whether a similar
approach is possible for $n\geq 15$, either in characteristic 2
or in characteristic not 2.

\section{Main result}

\begin{theorem}
\label{spin}
Let $k$ be a field. For every integer
$n$ at least 15,
the essential dimension of the split group $\Spin(n)$ over $k$
is given by:
$$\ed_2(\Spin(n))=\ed(\Spin(n))=\begin{cases}
2^{n-1}-n(n-1)/2 & \text{if $n$ is odd;}\\
2^{(n-2)/2}-n(n-1)/2 &\text{if }n\equiv 2\pmod{4};\\
2^{(n-2)/2}+2^m-n(n-1)/2 &\text{if }n\equiv 0\pmod{4},
\end{cases}$$
where $2^m$ is the largest power of 2 dividing $n$.
\end{theorem}

\begin{proof}
For $k$ of characteristic 0, this was proved by Chernousov
and Merkurjev, sharpening the results of Brosnan,
Reichstein, and Vistoli \cite[Theorem 2.2]{CM}.
Their argument works in any characteristic not 2, using
the results of Garibaldi and Guralnick for the upper bounds
\cite{GG}.
Namely, Garibaldi and Guralnick showed that for any field $k$
and any $n$ at least 15, $\Spin(n)$ acts generically
freely on the spin representation for $n$ odd, on each of the two half-spin
representations if $n\equiv 2\pmod{4}$, and on the direct sum
of a half-spin representation and the standard representation
if $n\equiv 0\pmod{4}$. Moreover,
for $n$ at least 20 with
$n\equiv 0\pmod{4}$, $\HSpin(n)=\Spin(n)/\mu_2$ (the quotient
different from $O^+(n)$) acts generically freely
on a half-spin representation \cite[Theorem 1.1]{GG}.

It remains to consider a field $k$ of characteristic 2.
Garibaldi and Guralnick's result gives the desired upper bound
in most cases. Namely, for $n$ odd and at least 15,
the spin representation has dimension $2^{(n-1)/2}$, and so
$\ed(\Spin(n))\leq 2^{(n-1)/2}-\dim(\Spin(n))=2^{(n-1)/2}-n(n-1)/2$.
For $n\equiv 2\pmod{4}$, the half-spin representations have
dimension $2^{(n-2)/2}$, and so $\ed(\Spin(n))
\leq 2^{(n-2)/2}-n(n-1)/2$. For $n=16$, since the spin group
acts generically freely on the direct sum of a half-spin representation
and the standard representation, $\ed(\Spin(n))\leq 2^{(n-2)/2}+n-n(n-1)/2$
($=24$). 

For $n$ at least 20 and divisible by 4, the optimal upper bound
requires more effort. The following argument is
modeled on Chernousov and Merkurjev's characteristic
zero argument \cite[Theorem 2.2]{CM}.
Namely, consider the map of exact sequences of $k$-group
schemes:
$$\xymatrix@C-10pt@R-10pt{
1\ar[r] & \mu_2 \ar[r]\ar[d]_{=} & \Spin(n) \ar[r]\ar[d] &
\HSpin(n) \ar[r]\ar[d] & 1\\
1\ar[r] & \mu_2 \ar[r] & O^{+}(n) \ar[r] &
\PGO^{+}(n) \ar[r] & 1.
}$$
Since $\HSpin(n)$ acts generically freely on a half-spin
representation, which has dimension $2^{(n-2)/2}$, we have
$\ed(\HSpin(n))\leq 2^{(n-2)/2}-n(n-1)/2$.

By Chernousov-Merkurjev or independently L\"otscher, for
any normal subgroup scheme $C$ of an affine group scheme $G$
over a field $k$,
$$\ed(G)\leq \ed(G/C)+\max\,\ed\,[E/G],$$
where the maximum runs over all field extensions $F$ of $k$
and all $G/C$-torsors $E$ over $F$ \cite[Proposition 2.1]{CM},
\cite[Example 3.4]{Lotscher}. Thus $[E/G]$ is a gerbe over $F$ banded
by $C$. 

Identifying $H^2(K,\mu_p)$ with the 
$p$-torsion in the Brauer group of $K$, we can talk about the index
of an element of $H^2(K,\mu_p)$, meaning the degree of the corresponding
division algebra over $K$. 
For a prime number $p$ and
a nonzero element $E$ of $H^2(K,\mu_p)$ over a field $K$,
the essential dimension (or also the $p$-essential dimension)
of the corresponding $\mu_p$-gerbe over $K$ is equal to the index
of $E$, by Karpenko and Merkurjev \cite[Theorems 2.1 and 3.1]{KM}.

By the diagram above, for any field $F$ over $k$, the image of the connecting
map
$$H^1(F, \HSpin(n))\arrow H^2(F,\mu_2)\subset \Br(F)$$
is contained in the image of the other connecting map
$$H^1(F, PGO^{+}(n))\arrow H^2(F,\mu_2)\subset \Br(F).$$
In the terminology of the Book of Involutions,
the image of the latter map consists of the classes $[A]$
of all central simple $F$-algebras $A$ of degree $n$ with
a quadratic pair $(\sigma,f)$ of trivial discriminant
\cite[section 29.F]{KMRT}. Any torsor for $PGO^{+}(n)$
is split by a field extension of degree a power of 2,
by reducing to the corresponding fact about quadratic forms.
So $\ind(A)$ must be a power of 2, but it also divides $n$,
and so $\ind(A)\leq 2^m$, where
$2^m$ is the largest power of 2 dividing $n$. We conclude
that
\begin{align*}
\ed(\Spin(n))&\leq \ed(\HSpin(n))+2^m\\
&\leq 2^{(n-2)/2}-n(n-1)/2+2^m.
\end{align*}
This completes the proof of the upper bound in Theorem
\ref{spin}.

We now prove the corresponding lower bound for the 2-essential
dimension of the spin group over a field $k$ of characteristic 2.
Since $\ed_2(\Spin(n))\leq \ed(\Spin(n))$, this will imply
that the 2-essential dimension and the essential dimension
are both equal to the number given in Theorem \ref{spin}.

Write $O(2r)$ for the orthogonal group of the quadratic
form $x_1x_2+x_3x_4+\cdots+x_{2r-1}x_{2r}$ over $k$,
and $O(2r+1)$ for the orthogonal group
of $x_1x_2+x_3x_4+\cdots+x_{2r-1}x_{2r}+x_{2r+1}^2$.
Then we have an inclusion $O(2r)\subset O(2r+1)$. Note that
$O(2r)$ is smooth over $k$, with $O(2r)/O^+(2r)\cong \Z/2$.
The group scheme $O(2r+1)$ is not smooth over $k$, but it contains
a smooth connected subgroup $O^+(2r+1)$ with $O(2r+1)\cong O^+(2r+1)
\times \mu_2$. It follows that $O(2r)$ is contained
in $O^+(2r+1)$. Using the subgroup $\Z/2\times \mu_2$ of $O(2)$,
we have a $k$-subgroup scheme 
$K:=(\Z/2\times \mu_2)^r\subset O(2r)\subset O^+(2r+1)$.
Let $G$ be the inverse image of $K$ in the double cover
$\Spin(2r+1)$ of $O^+(2r+1)$. Thus
$G$ is a central extension
$$1\arrow \mu_2\arrow G\arrow (\Z/2)^{r}\times (\mu_2)^r\arrow 1.$$
(Essentially the same ``finite Heisenberg group scheme'' appeared
in the work of Mumford and Sekiguchi on abelian varieties
\cite[Appendix A]{Sekiguchi}.)

To describe the structure of $G$ in more detail, think
of $K=(\mu_2)^r$ as the 2-torsion subgroup scheme of
a fixed maximal torus $T_{SO}\cong (G_m)^r$ in $O^+(2r+1)$.
The chararacter group of $T_{SO}$ is the free abelian group
$\Z\{x_1,\ldots,x_r\}$, and the Weyl group $W=N(T_{SO})/T_{SO}$
of $O^+(2r+1)$
is the semidirect product $S_r\ltimes (\Z/2)^r$.
Here $S_r$ permutes the characters $x_1,\ldots,x_r$ of $T_{SO}$,
and the subgroup $E_r=(\Z/2)^r$ of $W$, with generators
$\epsilon_1,\ldots,\epsilon_r$, acts by: $\epsilon_i$
changes the sign of $x_i$ and fixes $x_j$ for $j\neq i$.
The character group of $K=T_{SO}[2]$
is $\Z/2\{x_1,\ldots,x_r\}$. The group $E_r$ centralizes
$K$, and the group $(\Z/2)^r\times (\mu_2)^r\subset O^+(2r+1)$
above is $E_r\times K$.

Let $L$ be the inverse image of $K$ in $\Spin(2r+1)$,
which is contained in a maximal torus $T$ of $\Spin(2r+1)$,
the inverse image of $T_{SO}$.
The character group $X^*(T)$ is
$$\Z\{x_1,\ldots,x_r,A\}/(2A=x_1+\cdots+x_r).$$
Therefore,
the character group $X^*(L)$ is
$$\Z\{x_1,\ldots,x_r,A\}/(2x_i=0,\; 2A=x_1+\cdots+x_r).$$
(Thus $X^*(L)$ is isomorphic to $(\Z/4)\times (\Z/2)^{r-1}$, and so $L$
is isomorphic to $\mu_4\times (\mu_2)^{r-1}$.)
The Weyl group $W$ of $\Spin(2r+1)$ is the same as that
of $O^+(2r+1)$, namely $S_r\ltimes E_r$.
In particular, the element $\epsilon_i$
of $E_r$ acts on $X^*(T)$ by changing the sign of $x_i$
and fixing $x_j$ for $j\neq i$, and hence it sends
$A$ to $A-x_i$.

The subset $S$ of $X^*(L)$ corresponding to characters of $L$
which are faithful on the center $\mu_2$ of $L$
is the complement of the subgroup $X^*(K)=\Z/2\{x_1,\ldots,
x_r\}$. Therefore, $S$ has order $2^r$. The group $E_r=(\Z/2)^r$
acts freely and transitively on $S$, since
$$\bigg( \prod_{i\in I}\epsilon_i\bigg) (A)=A-\sum_{i\in I}x_i$$
for any subset $I$ of $\{1,\ldots,r\}$.

The group $G=E_r\cdot L$ is the central extension considered
above. Now, let $V$ be a representation of $G$
over $k$ on which the center $\mu_2\subset L$ acts
faithfully by scalars. Then the restriction of $V$ to $L$
is fixed (up to isomorphism)
by the action of $E_r$ on $X^*(L)$. By the previous
paragraph, the $2^r$ 1-dimensional representations of $L$ that
are nontrivial on the center $\mu_2$ all occur
with the same multiplicity in $V$. Therefore, $V$ has dimension
a multiple of $2^r$. This bound is optimal,
since the spin representation $W$ of $\Spin(2r+1)$ has dimension $2^{r}$
over $k$, and the center $\mu_2$ acts faithfully by scalars
on $W$.

We use the following result of Merkurjev's \cite[Theorem 5.2]{Merkmax},
\cite[Remark 4.5]{KM}. (The first reference covers the case
of the group scheme $\mu_p$ in characteristic $p$, as needed here.)

\begin{theorem}
Let $k$ be a field and $p$ be a prime number. Let $1\arrow \mu_p\arrow G
\arrow Q\arrow 1$ be a central extension of affine group schemes
over $k$. For a field extension $K$ of $k$, let
$\partial_K\colon H^1(K,Q)\arrow H^2(K,\mu_p)$
be the boundary homomorphism in fppf cohomology.
Then the maximal value of the index
of $\partial_K(E)$, as $K$ ranges over all field extensions of $k$
and $E$ ranges over all $Q$-torsors over $K$, is equal to the greatest
common divisor of the dimensions of all representations of $G$
on which $\mu_p$ acts by its standard representation.
\end{theorem}

As mentioned above, for a prime number $p$ and
a nonzero element $E$ of $H^2(K,\mu_p)$ over a field $K$,
the essential dimension (or also the $p$-essential dimension)
of the corresponding $\mu_p$-gerbe over $K$ is equal to the index
of $E$.

Finally, consider a central extension $1\arrow \mu_p\arrow G
\arrow Q\arrow 1$ of finite group schemes over a field $k$.
Generalizing an argument of Brosnan-Reichstein-Vistoli,
Karpenko and Merkurjev showed that the $p$-essential dimension
of $G$ (and hence the essential dimension of $G$) is at least
the $p$-essential dimension of the $\mu_p$-gerbe over $K$
associated to any $Q$-torsor over any field $K/k$ \cite[Theorem 4.2]{KM}.
By the analysis above of representations
of the finite subgroup scheme $G$ of $\Spin(2r+1)$ over a field $k$
of characteristic 2, we find that
$\ed_2(G)\geq 2^{r}$. For a closed subgroup scheme $G$
of a group scheme $L$ over a field $k$
and any prime number $p$, we have $\ed_p(L)+\dim(L)\geq \ed_p(G)
+\dim(G)$ \cite[Corollary 4.3]{Merkess} (which covers the case
of fppf torsors for non-smooth group schemes, as needed here).
Applying this to the
subgroup scheme $G$ of $\Spin(2r)$, we conclude that
$\ed_2(\Spin(2r+1))\geq
2^{r}-\dim(\Spin(2r+1))=2^{r}-r(2r+1)$. Combining this
with the upper bound discussed above, we have
$$\ed(\Spin(2r+1))=\ed_2(\Spin(2r+1))=2^{r}-r(2r+1)$$
for $r\geq 7$.

The proof of the lower bound for $\ed_2(\Spin(2r))$ when
$r$ is odd is similar. The intersection of the subgroup
$K=(\mu_2\times \Z/2)^r\subset O(2r)$ with $O^+(2r)$
is $K_1\cong (\mu_2)^r\times (\Z/2)^{r-1}$, where $(\Z/2)^{r-1}$
denotes the kernel of the sum $(\Z/2)^r\arrow \Z/2$.
As a result, the double cover $\Spin(2r)$ contains a subgroup $G_1$
which is a central extension
$$1\arrow \mu_2\arrow G_1\arrow (\Z/2)^{r-1}\times (\mu_2)^r\arrow 1.$$
In this case, an argument analogous to the one for $G$ shows that
every representation of $G_1$ on which the center $\mu_2$ acts
by its standard representation has dimension a multiple of $2^{r-1}$
(rather than $2^r$). The argument is otherwise identical to the
argument for $\Spin(2r+1)$, and we find that
$\ed_2(\Spin(2r))\geq 2^{r-1}-r(2r-1)$. For $r$ odd at least 9,
this agrees with the lower bound found earlier, which proves
the theorem on $\Spin(n)$ for $n\equiv 0\pmod{4}$.

It remains to show that for $n$ a multiple of 4, with $2^m$ the largest
power of 2 dividing $n$, we have
$$\ed_2(\Spin(n))\geq 2^{(n-2)/2}+2^m-n(n-1)/2.$$
The argument follows that of Merkurjev in characteristic not 2
\cite[Theorem 4.9]{Merkess}. 

Namely, for $n$ a multiple of 4, the center $C$ of $G:=\Spin(n)$
is isomorphic to $\mu_2\times \mu_2$, and $H:=G/C$ is the group
$PGO^{+}(n)$. An $H$-torsor over a field $L$ over $k$ is equivalent
to a central simple algebra $A$ of degree $n$ over $L$ with a quadratic
pair $(\sigma,f)$ and with trivialized discriminant, meaning
an isomorphism from the center of the Clifford algebra
$C(A,\sigma,f)$ to $L\times L$
\cite[section 29.F]{KMRT}. The image of the homomorphism
from $C^*\cong (\Z/2)^2$ to the Brauer group of $L$
is equal to $\{0,[A],[C^+],[C^-]\}$, where $C^+$ and $C^-$
are the simple components of the Clifford algebra; each is 
a central simple algebra of degree $2^{(n-2)/2}$ over $L$.
By Merkurjev, there is a field $L$ over $k$ and an $H$-torsor $E$
over $L$ such that $\ind(C^+)=\ind(C^-)=2^{(n-2)/2}$
and $\ind(A)=2^m$ \cite[section 4.4 and Theorem 5.2]{Merkmax}.
We use the following result
\cite[Example 3.7]{Merkess}:

\begin{lemma}
Let $L$ be a field, $p$ a prime number, and $r$ a natural
number. Let $C$ be the group scheme $(\mu_p)^r$,
and let $Y$ be a $C$-gerbe over $L$.
Then the $p$-essential dimension of $Y$, and also the essential
dimension of $Y$, is the minimum,
over all bases $u_1,\ldots,u_r$ for $C^*$,
of $\sum_{i=1}^r \ind(u_i(Y))$.
\end{lemma}

It follows that the 2-essential dimension of the $(\mu_2)^2$-gerbe
$E/G$ over $L$ associated to the $H$-torsor $E$ above
is
$$\ed_2(E/G)=\ind(A)+\ind(C^+)=2^{(n-2)/2}+2^m.$$
It follows that
\begin{align*}
\ed(\Spin(n))&\geq \ed_2(\Spin(n))\\
&\geq \ed_2(E/G)-\dim(G/C)\\
&=2^{(n-2)/2}+2^m-n(n-1)/2.
\end{align*}
\end{proof}

\section{Low-dimensional spin groups}

Rost and Garibaldi determined the essential dimension of the spin groups
$\Spin(n)$ with $n\leq 14$
in characteristic not 2
\cite[Table 23B]{Garibaldi}. It should be possible
to compute the essential dimension of low-dimensional
spin groups in characteristic 2 as well. The following section
carries this out for $\Spin(n)$ with $n\leq 10$. We find 
that in this range (as for $n\geq 15$),
the essential dimension of the spin group
is the same in characteristic 2 as in characteristic not 2,
unlike what happens for $O(n)$ and $SO(n)$.

For $n\leq 10$, we give group-theoretic proofs which
work almost the same way
in any characteristic, despite the distinctive features
of quadratic forms in characteristic 2.

\begin{theorem}
\label{lowdim}
For $n\leq 10$, the essential dimension, as well as the 2-essential
dimension, of the split group $\Spin(n)$
over a field $k$ of any characteristic is given by:
$$\begin{array}{rr}
n& \ed(\Spin(n))\\
\leq 6 & 0\\
7 & 4\\
8 & 5\\
9 & 5\\
10 & 4
\end{array}$$
\end{theorem}

\begin{proof}
As discussed above, it suffices to consider the case
of a field $k$ of characteristic 2. For $n\leq 6$,
every $\Spin(n)$-torsor over a field is trivial,
for example by the exceptional
isomorphisms $\Spin(3)\cong SL(2)$, $\Spin(4)\cong SL(2)\times SL(2)$,
$\Spin(5)\cong Sp(4)$, and $\Spin(6)\cong SL(4)$. It follows
that $\ed(\Spin(n))=0$ for $n\leq 6$.

We first recall some general definitions.
For a field $k$ of characteristic $p>0$,
let $H^{i,j}(k)$ be the etale motivic cohomology group
$H^{i}_{\et}(k,\Z/p(j))$, or equivalently
$$H^i_{\et}(k,\Z/p(j))\cong H^{i-j}_{\et}(k,\Omega^j_{\log}),$$
where $\Omega^j_{\log}$ is the subgroup of the group $\Omega^j$
of differential forms on the separable closure $k_s$ over $\F_p$
spanned by products $(da_1/a_1)\wedge\cdots\wedge(da_j/a_j)$
with $a_1,\ldots,a_j\in k_s^*$ \cite{GL}. The group $H^{i,j}(k)$
is zero except when $i$ equals
$j$ or $j+1$, because $k$ has $p$-cohomological dimension at most 1
\cite[section II.2.2]{SerreCohoGal}. The symbol $\{ a_1,\ldots,a_{n-1},b]$
denotes the element of $H^{n,n-1}(k)$ which is the product
of the elements $a_i\in k^*/(k^*)^p\cong H^{1,1}(k)$
and $b\in k/\{a^p-a:a\in k\}\cong H^{1,0}(k)$.

Also, for a field $k$ of characteristic 2, let $W(k)$ denote
the Witt ring of symmetric bilinear forms over $k$,
and let $I_q(k)$ be the Witt group of nondegenerate quadratic forms
over $k$. (By the conventions in section \ref{notation},
$I_q(k)$ consists only of even-dimensional
forms.)  Then $I_q(k)$ is a module over $W(k)$ via tensor
product \cite[Lemma 8.16]{EKM}.
Let $I$ be the kernel of the homomorphism $\text{rank}\colon W(k)\arrow
\Z/2$, and let 
$$I^m_q(k)=I^{m-1}\cdot I_q(k),$$
following \cite[p.~53]{EKM}. To motivate the notation, observe that
the class of an $m$-fold quadratic Pfister form
$\langle \langle a_1,\ldots,a_{m-1},b]]$ lies in $I^m_q(k)$.
By definition, for $a_1,\ldots,a_{m-1}$ in $k^*$ and $b$ in $k$,
$\langle\langle a_1,\ldots,a_{m-1},b]]$ is the quadratic
form $\langle\langle a_1\rangle\rangle_b\otimes \cdots\otimes
\langle\langle a_{m-1}\rangle\rangle_b
\otimes \langle\langle b]]$ of dimension $2^m$,
where $\langle\langle a\rangle\rangle_b$
is the bilinear form $\langle 1,a\rangle$
and $\langle\langle b]]$ is the quadratic form $[1,b]=x^2+xy+by^2$.

In analogy with the Milnor conjecture, Kato
proved the isomorphism
$$I^m_q(F)/I^{m+1}_q\cong H^{m,m-1}(F)$$
for every field $F$ of characteristic 2 \cite[Fact 16.2]{EKM}.
The isomorphism takes the quadratic Pfister form $\langle\langle
a_1,\ldots,a_{m-1},b]]$ to the symbol $\{ a_1,\ldots,a_{m-1},b]$.
(For this paper, it would suffice to have Kato's
homomorphism, without knowing that it is an isomorphism.)

We will use the following standard approach
to bounding the essential dimension of a group.

\begin{lemma}
\label{surj}
Let $G$ be an affine group scheme of finite type over a field $k$.
Suppose that $G$ acts on a $k$-scheme $Y$ with a nonempty
open orbit $U$. Suppose that for every $G$-torsor $E$
over an infinite field $F$ over $k$, the twisted form
$(E\times Y)/G$ of $Y$ over $F$ has a Zariski-dense set
of $F$-points. Finally, suppose that $U$ has a $k$-point $x$,
and let $N$ be the stabilizer $k$-group scheme of $x$ in $G$.
Then
$$H^1(F,N)\arrow H^1(F,G)$$
is surjective for every infinite field $F$ over $k$
(or for every field $F$ over $k$, if $G$ is smooth and connected).
As a result, $\ed_k(G)\leq \ed_k(N)$.
\end{lemma}

The proof is short, the same as that of \cite[Theorem 9.3]{Garibaldi}.
(Note that even if $k$ is finite, we get the stated upper bound
for the essential dimension of $G$: a $G$-torsor
over a finite field $F$ that contains $k$ causes no problem,
because $F$ has transcendence degree 0 over $k$.)
If $G$ is smooth and connected, then $H^1(F,G)$ is in fact
trivial for every finite field $F$ that contains $k$, by Lang \cite{Lang};
that implies the statement in the theorem
that $H^1(F,N)\arrow H^1(F,G)$ is
surjective for {\it every }field $F$ over $k$.

The assumption about a Zariski-dense set of rational points holds,
for example, if $Y$ is a linear representation $V$ of $G$,
or if $Y$ is the associated projective space $P(V)$ to a representation,
or (as we use later) a product $P(V)\times P(W)$.

We use Garibaldi and Guralnick's calculation
of the stabilizer group scheme of a general
$k$-point in the spin (for $n$ odd) or a half-spin
(for $n$ even) representation $W$ of the split group
$\Spin(n)$, listed in Table 1 here
\cite[Table 1]{GG}.
Here $\Spin(n)$ has an open orbit on the projective space $P(W)$
of lines in $W$ if $n\leq 12$ or $n=14$,
and an open orbit on $W$ if $n=10$. (To be precise, we will use
that even if $k$ is finite, there is a $k$-point in the open
orbit for which the stabilizer $k$-group scheme
is the {\it split }group listed in the table.)
\begin{table}
\centering
\begin{tabular}{ccc}
$n$ &$ \text{char }k\neq 2$ & $\text{char }k=2$\\
6 & $SL(3)\cdot (G_a)^3$ & same\\
7 & $G_2$ & same\\
8 & $\Spin(7)$ & same\\
9 & $\Spin(7)$ & same\\
10 & $\Spin(7)\cdot (G_a)^8$ & same\\
11 & $SL(5)$ & $\Z/2\ltimes SL(5)$\\
12 & $SL(6)$ & $\Z/2\ltimes SL(6)$\\
13 & $SL(3)\times SL(3)$ & $\Z/2\ltimes (SL(3)\times SL(3))$\\
14 & $G_2\times G_2$ & $\Z/2\ltimes(G_2\times G_2)$
\end{tabular}
\caption{Generic stabilizer of spin (or half-spin) representation
of $\Spin(n)$}
\end{table}

We now begin to compute the essential dimension of
the split group $G=\Spin(7)$ over a field $k$
of characteristic 2.
Let $W$ be the 8-dimensional spin representation
of $G$.
Then $G$ has an open orbit on the projective space 
$P(W)$ of lines in $W$. By Table 1, there is a $k$-point $x$ in $W$
whose image in $P(W)$ is in the open
orbit such that the stabilizer of $x$ in $G$ is the split
exceptional group $G_2$.
Since $G$ preserves a quadratic form
on $W$, the stabilizer $H$ of the corresponding $k$-point
in $P(W)$ is at most $G_2\times \mu_2$. In fact, $H$ is equal
to $G_2\times \mu_2$, because the center $\mu_2$ of $G$
acts trivially on $P(W)$.

By Lemma \ref{surj}, the inclusion $G_2\times \mu_2\inj G$
induces a surjection
$$H^1(F,G_2\times\mu_2)\arrow H^1(F,G)$$
for every field $F$ over $k$.
Over any field $F$, $G_2$-torsors up to isomorphism can be identified
with 3-fold quadratic Pfister forms $\langle\langle a_1,a_2,b]]$
(with $a_1,a_2\in F^*$ and $b\in F$),
and so $G_2$ has essential dimension 3 \cite[Th\'eor\`eme 11]{SerreCohoGal}.
Since $\mu_2$
has essential dimension 1, the surjectivity above
implies that $G=\Spin(7)$ has essential dimension at most 4.

Next, a $G$-torsor determines two quadratic forms of dimension 8.
Besides the obvious homomorphism $\chi_1\colon G\inj \Spin(8)\arrow SO(8)$
(which is trivial on the center $\mu_2$ of $G$), we have
the spin representation $\chi_2\colon G\arrow SO(8)$, on which $\mu_2$
acts faithfully by scalars. Thus a $G$-torsor $u$ over a field $F$
over $k$ determines two quadratic forms of dimension 8
over $F$, which we call $q_1$ and $q_2$.

To describe these quadratic forms in more detail, use that
every $G$-torsor comes from a torsor for $G_2\times \mu_2$.
The two homomorphisms $G_2\inj G\arrow SO(8)$ (via $\chi_1$
and $\chi_2$) are both conjugate to the standard inclusion.
Also, $\chi_1$ is trivial on the $\mu_2$ factor, while $\chi_2$
acts faithfully by scalars on the $\mu_2$ factor. It follows 
that $q_1$ is a quadratic Pfister form, $\langle\langle
a,b,c]]$ (the form associated to a $G_2$-torsor), while $q_2$
is a scalar multiple of that form, $d\langle\langle a,b,c]]$.

Therefore, a $G$-torsor $u$ canonically determines a 4-fold
quadratic Pfister form,
$$q_1+q_2=\langle\langle d,a,b,c]].$$
Define $f_4(u)$ to be the associated element of $H^{4,3}(F)$,
$$f_4(u)=\{ d,a,b,c].$$
By construction, this is well-defined and an invariant of $u$.
This invariant is normalized (zero on the trivial $G$-torsor)
and not zero. (By considering the subgroup $G_2\times \mu_2\subset
\Spin(7)$, where there is a $G_2\times \mu_2$-torsor associated
to any elements $a,b,d$ in $F^*$ and $c$ in $F$, we see that $a,b,c,d$ can
be chosen arbitrarily. By taking $F$ to be the rational
function field $k(a,b,c,d)$, we see that the element
$f_4(u)=\{ d,a,b,c]$ of $H^{4,3}(F)$ can be nonzero. For that, one can
use the computation of $H^{n,n-1}$ of a rational function field
by Izhboldin \cite{Izhboldinrat}.)

Therefore, $G=\Spin(7)$ has essential dimension at least 4.
The opposite inequality was proved above, and so
$\Spin(7)$ has essential dimension equal to 4. Since the lower
bound is proved by constructing a mod 2 cohomological invariant,
this argument also shows that $\Spin(7)$ has 2-essential dimension
equal to 4. For the same reason, the computations of essential
dimension below (for $\Spin(n)$ with $8\leq n\leq 10$)
also give the 2-essential dimension.

Next, we turn to $\Spin(8)$. At first,
let $G=\Spin(2r)$ for a positive integer $r$ over a field
$k$ of characteristic 2.
Let $V$ be the standard $2r$-dimensional
representation of $G$. Then $G$ has an open orbit
in the projective space $P(V)$ of lines in $V$. The stabilizer
$k$-group scheme $H$ of a general $k$-point in $P(V)$ is conjugate
to $\Spin(2r-1)\cdot Z$, where $Z$ is the center of $\Spin(2r)$,
with $\Spin(2r-1)\cap Z=\mu_2$.
(In more detail, a general line in $V$ is spanned by a vector $x$
with $q(x)\neq 0$, where $q$ is the quadratic form on $V$.
Then the stabilizer of $x$ in $SO(V)$ is isomorphic to $SO(S)$,
where $S:=x^{\perp}$ is a hyperplane in $V$ on which $q$ restricts
to a nonsingular quadratic form of dimension $2r-1$,
with $S^{\perp}$ equal to the line $k\cdot x\subset S$.)
Here
$$Z\cong \begin{cases}
\mu_2\times\mu_2 &\text{if $r$ is even}\\
\mu_4 &\text{if $r$ is odd.}
\end{cases}$$
In particular, if $r$ is even, then $H\cong \Spin(2r-1)\times\mu_2$.
Thus, for $r$ even, the inclusion $\Spin(2r-1)\times\mu_2\inj G$
induces a surjection
$$H^1(F,\Spin(2r-1)\times\mu_2)\arrow H^1(F,G)$$
for every field $F$ over $k$, by Lemma \ref{surj}.

It follows that, for $r$ even, the essential dimension of $\Spin(2r)$
is at most 1 plus the essential dimension
of $\Spin(2r-1)$. Since $\Spin(7)$ has essential
dimension 4, $G=\Spin(8)$ has essential
dimension at most 5.

Before proving that equality holds,
let us analyze $G$-torsors in more detail. We know
that $H^1(F,\Spin(7)\times\mu_2)\arrow H^1(F,G)$ is onto,
for all fields $F$ over $k$.
Also, we showed earlier that $H^1(F,G_2\times\mu_2)\arrow H^1(F,\Spin(7))$
is surjective. Therefore,
$$H^1(F,G_2\times\mu_2\times\mu_2)\arrow H^1(F,G)$$
is surjective for all fields $F$ over $k$, where $Z=\mu_2\times\mu_2$
is the center of $G$. As discussed earlier,
$G_2$-torsors up to isomorphism
can be identified with 3-fold quadratic Pfister forms.
It follows that
every $G$-torsor is associated to some
3-fold quadratic Pfister form $\langle\langle a,b,c]]$
and some elements $d,e$ in $F^*$, which yield elements
of $H^1(F,\mu_2)=F^*/(F^*)^2$.

Next, observe that a $G$-torsor determines several quadratic forms.
Besides the obvious double covering $\chi_1\colon G\arrow SO(8)$,
the two half-spin representations of $G$ give two other
homomorphisms $\chi_2,\chi_3\colon G\arrow SO(8)$.
(These three homomorphisms can
be viewed as the quotients of $G$ by the three $k$-subgroup schemes
of order 2 in $Z$. They are permuted 
by the group $S_3$
of ``triality'' automorphisms of $G$.)
Thus a $G$-torsor $u$ over a field $F$ over $k$ determines
three quadratic forms of dimension 8, which we call
$q_1,q_2,q_3$.

To describe how these three quadratic forms are related,
use that every $G$-torsor comes from a torsor for $G_2\times\mu_2\times
\mu_2$. The three homomorphisms $G_2\arrow G\arrow SO(8)$
(via $\chi_1$, $\chi_2$, and $\chi_3$)
are all conjugate to the standard inclusion, whereas the three
homomorphisms send $\mu_2\times \mu_2$ to the center
$\mu_2\subset SO(8)$ by the three possible surjections.
It follows that the three quadratic forms can be written
as $q_1=d\langle\langle a,b,c]]$,
$q_2=e\langle\langle a,b,c]]$, and $q_3=de\langle\langle a,b,c]]$.

Note that a scalar multiple of a quadratic Pfister form,
$q=d\langle\langle a_1,\ldots,
a_{m-1},b]]$ (as a quadratic form up to isomorphism),
uniquely determines the associated quadratic Pfister form
$q_0=\langle\langle a_1,\ldots,a_{m-1},b]]$ up to isomorphism.
(Proof: it suffices to show that if $q$ and $r$ are $m$-fold
quadratic Pfister forms over $F$
with $aq\cong r$ for some $a$ in $F^*$, then $q\cong r$. Since $r$
takes value 1, so does $aq$, and so $q$ takes value $a^{-1}$.
But then $a^{-1}q\cong q$ by the multiplicativity of quadratic
Pfister forms \cite[Corollary 9.9]{EKM}. Therefore, $r\cong aq\cong q$.)

We now define an invariant for $G=\Spin(8)$ over $k$ with values
in $H^{5,4}$. Given a $G$-torsor $u$ over a field $F$ over $k$,
consider the three associated quadratic forms $q_1,q_2,q_3$
as above. By the previous paragraph, $q_1=d\langle \langle a,b,c]]$
determines the quadratic Pfister form
$q_0=\langle\langle a,b,c]]$. So $u$ determines 
the 5-fold quadratic Pfister form
$$q_0+q_1+q_2+q_3=\langle\langle d,e,a,b,c]].$$
The associated class
$$f_5(u)=\{d,e,a,b,c]\in H^{5,4}(F)$$
is therefore an invariant of $u$.

The invariant $f_5$ is normalized and not 0, as shown by considering
the subgroup $G_2\times Z\subset G=\Spin(8)$, where $Z=\mu_2\times \mu_2$:
there is a $G_2\times Z$-torsor associated to any elements
$a,b,d,e$ in $F^*$ and $c$ in $F$, and $f_5$ of the associated $G$-torsor
is $\{d,e,a,b,c]$ in $H^{5,4}(F)$. Therefore, $G$ has essential
dimension at least 5. Since the opposite inequality was proved
above, $G=\Spin(8)$ has essential dimension over $k$ equal to 5.

Next, let $G=\Spin(9)$ over a field $k$ of characteristic 2.
Let $W$ be the spin representation of $G$, of dimension 16,
corresponding to a homomorphism $G\arrow SO(16)$. (A reference
for the fact that this self-dual representation is orthogonal
in characteristic 2, as in other characteristics,
is \cite[Theorem 9.2.2]{GN}.)
By Table 1, $G$ has an open orbit on the space
$P(W)$ of lines in $W$, and the stabilizer in $G$ of a general
$k$-point in $W$ is conjugate to $\Spin(7)$. (This is
not the standard inclusion of $\Spin(7)$ in $\Spin(9)$,
but rather a lift of the spin representation $\chi_2\colon
\Spin(7)\arrow SO(8)$ to $\Spin(8)$ followed by the standard
inclusion $\Spin(8)\inj \Spin(9)$.
In particular, the image of $\Spin(7)$ does not contain the center $\mu_2$
of $G=\Spin(9)$.)
Since $G$ preserves a quadratic form on $W$, it follows
that the stabilizer in $G$ of a general $k$-point in $P(W)$
is conjugate to $\Spin(7)\times \mu_2$, where $\mu_2$
is the center of $\Spin(9)$ (which acts faithfully by scalars
on $W$). Therefore, by Lemma \ref{surj},
the inclusion of $\Spin(7)\times \mu_2$
in $G=\Spin(9)$ induces a surjection
$$H^1(F,\Spin(7)\times \mu_2)\arrow H^1(F,G)$$
for every field $F$ over $k$.

Since $\Spin(7)$ has essential dimension 4 over $k$ as shown above,
$G=\Spin(9)$ has essential dimension at most $4+1=5$.

Next, a $G$-torsor determines several quadratic forms. Besides
the obvious homomorphism $R\colon
G\inj \Spin(10)\arrow SO(10)$,
we have the spin representation $S\colon G\arrow SO(16)$.
Thus a $G$-torsor over a field $F$ over $k$ determines
a quadratic form $r$ of dimension 10 and a quadratic form
$s$ of dimension 16.

To describe how these forms are related, use that every $G$-torsor
comes from a torsor for the subgroup $\Spin(7)\times \mu_2$
described above. The restriction of $R$
to the given subgroup $\Spin(7)$ is the composition
of the spin representation $\chi_2\colon \Spin(7)\arrow SO(8)$
with the obvious inclusion $SO(8)\inj SO(10)$. The restriction
of $S$ to the given subgroup
$\Spin(7)$ is the direct sum of the standard representation
$\chi_1\colon \Spin(7)\arrow SO(8)$ and the spin representation
$\chi_2\colon \Spin(7)\arrow SO(8)$. Finally, $R$
is trivial on the second factor
$\mu_2$ (the center of $G$), whereas $S$ acts faithfully by scalars
on $S$.

Now, let $(u_1,e)$ be a $\Spin(7)\times \mu_2$-torsor over $k$, where
$u_1$ is a $\Spin(7)$-torsor and $e$ is in $H^1(F,\mu_2)=F^*/(F^*)^2$,
which we lift to an element $e$ of $F^*$.
By the earlier analysis of the quadratic forms
associated to a $\Spin(7)$-torsor, the quadratic form associated
to $u_1$
via the standard representation $\chi_1\colon \Spin(7)\arrow SO(8)$
is a 3-fold quadratic Pfister form $\langle\langle a,b,c]]$, while
the quadratic form associated to $u_1$ via the spin representation
$\chi_2\colon \Spin(7)\arrow SO(8)$ is a multiple of the same form,
$d\langle\langle a,b,c]]$.

By the analysis of representations two paragraphs back, it follows
that the quadratic form associated to $(u_1,e)$ via the representation
$R\colon G\arrow SO(10)$ is $r=H+d\langle\langle a,b,c]]$,
where $H$ is the hyperbolic plane.
Also, the quadratic form associated to $(u_1,e)$ via the representation
$S\colon G\arrow SO(16)$ is $s=e\langle\langle a,b,c]]+de
\langle\langle a,b,c]]$.

Next, $r$ determines the quadratic form
$r_0=d\langle\langle a,b,c]]$ by Witt cancellation \cite[Theorem 8.4]{EKM},
and that in turn
determines the quadratic Pfister form $q_0=\langle\langle a,b,c]]$
as shown above. Therefore, a $G$-torsor $u$
determines
the 5-fold quadratic Pfister form
$$q_0+r_0+s=\langle\langle d,e,a,b,c]]$$
up to isomorphism. 

Therefore, defining
$$f_5(u)=\{d,e,a,b,c]$$
in $H^{5,4}(F)$ yields an invariant of $u$. By our earlier
description of $\Spin(7)$-torsors, we can take $a,b,d,e$ to be
any elements of $F^*$ and $c$ any element of $F$.
Therefore, $f_5$ is a nonzero normalized
invariant of $G$ over $k$ with values in $H^{5,4}$. It follows
that $G$ has essential dimension at least 5. Since the opposite
inequality was proved earlier, $G=\Spin(9)$ over $k$ has essential
dimension equal to 5.

Finally, let $G=\Spin(10)$ over a field $k$ of characteristic 2.
Let $V$ be the 10-dimensional standard
representation of $G$, corresponding to the double
covering $G\arrow SO(10)$, and let $W$ be one of the 16-dimensional
half-spin representations of $G$, corresponding to 
a homomorphism $G\arrow SL(16)$.
(The other half-spin representation of $G$ is the dual $W^*$.)

As discussed above for any group $\Spin(2r)$, $G=\Spin(10)$
has an open orbit on $P(V)$, with generic stabilizer $\Spin(9)\cdot \mu_4$.
(Here $\mu_4$ is the center of $G$, which contains the center
$\mu_2$ of $\Spin(9)$.) Consider the action of $G$
on $P(V)\times P(W)\cong \P^9\times \P^{15}$. As discussed above,
$\Spin(9)$ (and hence $\Spin(9)\cdot\mu_4$)
has an open orbit on $P(W)$. As a result,
$G$ has an open orbit on $P(V)\times P(W)$. Moreover,
the generic stabilizer of $\Spin(9)$ on $P(W)$
is $\Spin(7)\times \mu_2$, where the inclusion $\Spin(7)\inj \Spin(9)$
is the composition of the spin representation $\Spin(7)\inj \Spin(8)$
with the standard inclusion into $\Spin(9)$; in particular, the image
does not contain the center $\mu_2$ of $\Spin(9)$.
Therefore, the generic stabilizer
of $\Spin(9)\cdot \mu_4\subset \Spin(10)$
on $P(W)$ is $\Spin(7)\times \mu_4$. We conclude that $G$
has an open orbit on $P(V)\times P(W)$, with generic stabilizer
$\Spin(7)\times \mu_4$. It follows that
$$H^1(F,\Spin(7)\times \mu_4)\arrow H^1(F,G)$$
is surjective for every field $F$ over $k$,
by Lemma \ref{surj}.

The image $H_2$ of the subgroup $H=\Spin(7)\times \mu_4\subset G$ in $SO(10)$
is $\Spin(7)\times\mu_2$, where $\Spin(7)$ is contained
in $SO(8)$ (and contains the center $\mu_2$ of $SO(8)$)
and $\mu_2$ is the center of $SO(10)$. In terms of the
subgroup $SO(8)\times SO(2)$ of $SO(10)$,
we can also describe $H_2$ as $\Spin(7)\times \mu_2$,
where $\Spin(7)$ is contained in $SO(8)$ and $\mu_2$ is contained
in $SO(2)$. Thus $H_2$ is contained in $\Spin(7)\times SO(2)$.
Therefore, $H$ is contained
in $\Spin(7)\times G_m\subset G=\Spin(10)$, where the multiplicative
group $G_m$ is the inverse
image in $G$ of $SO(2)\subset SO(10)$.
It follows that
$$H^1(F,\Spin(7)\times G_m)\arrow H^1(F,G)$$
is surjective for every field $F$ over $k$.
Since every $G_m$-torsor over a field is trivial,
$$H^1(F,\Spin(7))\arrow H^1(F,G)$$
is surjective for every field $F$ over $k$.

Here $\Spin(7)$ maps into $\Spin(8)$ by the spin
representation, and then $\Spin(8)\inj G=\Spin(10)$
by the standard inclusion. By the description above
of the 8-dimensional quadratic form associated to a $\Spin(7)$-torsor
by the spin representation, it follows that
the quadratic form associated to a $G$-torsor
is of the form $H+d\langle\langle a,b,c]]$.

Every 10-dimensional
quadratic form in $I^3_q$ over a field is associated
to some $G$-torsor. So we have given another proof that
every 10-dimensional quadratic form in $I^3_q$
is isotropic. This was proved in characteristic not 2
by Pfister, and it was
extended to characteristic 2 by Baeza
and Tits, independently \cite[pp.~129-130]{Baeza},
\cite[Theorem 4.4.1(ii)]{Tits}.

Since $\Spin(7)$ has essential dimension 4, the surjectivity
above implies that $G=\Spin(10)$ has essential
dimension at most 4. To prove equality, we define a nonzero
normalized invariant
for $G$ with values in $H^{4,3}$ by the same argument
used for $\Spin(7)$. Namely, a $G$-torsor $u$ over a field $F$
over $k$ determines
a 4-fold quadratic Pfister form $$\langle\langle d,a,b,c]]$$
up to isomorphism,
and hence the element
$$f_4(u)=\{d,a,b,c]$$
in $H^{4,3}(F)$.
This completes the proof
that $G=\Spin(10)$ over $k$ has essential dimension equal to 4.
As in the previous cases, since the lower bound is proved using
a mod 2 cohomological invariant, $G$ also has 2-essential
dimension equal to 4.
\end{proof}


\small \sc UCLA Mathematics Department, Box 951555,
Los Angeles, CA 90095-1555

totaro@math.ucla.edu
\end{document}